\documentclass{article}

\usepackage{graphicx,tikz,graphics,float,fp,amsthm,cite,multicol,subcaption,caption,amssymb,mathtools,hyphenat,multirow,longtable,array}
\usetikzlibrary{automata,positioning, decorations.markings,fixedpointarithmetic,arrows}
\usepackage{longtable,comment}
\usepackage[shortlabels]{enumitem}

\usetikzlibrary{external}
\newtheorem{theorem}{Theorem}[section]
\newtheorem{corollary}[theorem]{Corollary}

\newtheorem*{observation}{Observation}

\newtheorem{Lemma}[theorem]{Lemma}
\theoremstyle{definition}
\newtheorem*{example*}{Example}

\usetikzlibrary{decorations.pathreplacing,decorations.markings,arrows,calc,math,arrows.meta}
\newcolumntype{C}{>{$}c<{$}}
\newcolumntype{M}{>{$}p{10cm}<{$}}

\begin{document}

\author{Danny Dyer\\
	Department of Mathematics and Statistics\\
	St. John's Campus, Memorial University of Newfoundland\\
	St. John's, Newfoundland\\
	Canada\\
	{\tt dyer@mun.ca}
	\and
	Jared Howell\\
	School of Science and the Environment\\
	Grenfell Campus, Memorial University of Newfoundland\\
	Corner Brook, Newfoundland\\
	Canada\\
	{\tt jahowell@grenfell.mun.ca}
	\and
	Brittany Pittman\\
	Department of Mathematics and Statistics\\
	St. John's Campus, Memorial University of Newfoundland\\
	St. John's, Newfoundland\\
	Canada\\
	{\tt bep275@mun.ca}}

\title{The Watchman's Walk Problem on Directed Graphs}
\date{July 24, 2020}

\maketitle

\begin{abstract}
In a graph, a watchman's walk is a minimum closed dominating walk. Given a graph $G$ and a single watchman, the length of a watchman's walk in $G$ (the watchman number) is denoted by $w(G)$ and the typical goals of the watchman's walk problem is to determine $w(G)$ and find a watchman's walk in $G$. In this paper, we extend the watchman's walk problem to directed graphs. In a directed graph, we say that the watchman can only move to and see the vertices that are adjacent to him relative to outgoing arcs. That is, a watchman's walk is oriented and domination occurs in the direction of the arcs. The directed graphs this paper focuses on are families of tournaments and orientations of complete multipartite graphs. We give bounds on the watchman number and discuss its relationship to variants of the domination number.
\end{abstract}

\section{Introduction}

\tikzset{
  on each segment/.style={
    decorate,
    decoration={
      show path construction,
      moveto code={},
      lineto code={
        \path [#1]
        (\tikzinputsegmentfirst) -- (\tikzinputsegmentlast);
      },
      curveto code={
        \path [#1] (\tikzinputsegmentfirst)
        .. controls
        (\tikzinputsegmentsupporta) and (\tikzinputsegmentsupportb)
        ..
        (\tikzinputsegmentlast);
      },
      closepath code={
        \path [#1]
        (\tikzinputsegmentfirst) -- (\tikzinputsegmentlast);
      },
    },
  },
}

Given a graph $G$ and a single watchman, the watchman's walk problem looks at the scenario of that watchman traversing the graph in such a way that the route is a minimum closed dominating walk. Any such walk in $G$ is a \textit{watchman's walk} of $G$. The length of a watchman's walk in $G$ is denoted by $w(G)$; this is called the \textit{watchman number} of $G$. The typical goals of the watchman's walk problem is to determine $w(G)$ and find a closed dominating walk in $G$ that achieves this number. The watchman's walk was introduced by Hartnell, Rall, and Whitehead in 1998 in \cite{intro}.

Throughout this paper, we follow \cite{west}, with the exception of our use of $\deg(v)$ for the degree of a vertex $v$; basic definitions and notation can be found there.

The watchman's walk problem is a variation of the domination problem. Domination in graphs was formally introduced in 1958 by Berge in \cite{berge}. The domination problem in graphs aims to find a minimum dominating set; a set of vertices in a graph such that every vertex of the graph is either in that set or a neighbour of a vertex in that set. Variations of this problem include finding a minimum total dominating set, or a minimum connected dominating set. Directed domination in tournaments was first considered by Erd{\"o}s in \cite{erdos}. In \cite{irredundance}, Reid et al. consider the domination number and irredundance number of tournaments. They give upper bounds on the domination number of tournaments of small order. A survey of results on the domination number of graphs can be found in \cite{domfundamentals}.

If a dominating set induces a weakly connected subdigraph, we call it a \textit{weakly connected dominating set}. The size of a minimum weakly connected dominating set is denoted by $\gamma_{wc}(D)$. If a dominating set induces a strongly connected subdigraph, we call it a \textit{strongly connected dominating set}. The minimum size of such a set is denoted by $\gamma_{sc}(D)$. We call a minimum weakly connected dominating set a \textit{$\gamma_{wc}$-set}, and a minimum strongly connected dominating set a \textit{$\gamma_{sc}$-set}.

Finding a connected dominating set in an undirected graph, and a weakly or strongly connected dominating set in a directed graph is a problem similar to finding a watchman's walk. The vertices of a minimum closed dominating walk in an undirected graph are always a connected dominating set. In a directed graph, the vertices of a watchman's walk will be a strongly connected dominating set. However, in many directed graphs, these parameters and sets of vertices are not equal. In fact, a digraph may not contain any watchman's walks. If we consider the directed path in Figure \ref{directedpath}, the set of vertices $\{v_1,v_2,v_3,v_4,v_5\}$ is a minimum weakly connected dominating set, and $\gamma_{wc}(D)=5$. However, there are no strongly connected dominating sets or watchman's walks.

\begin{figure}[h]
    \begin{center}\vspace*{1.2cm}
   \begin{tikzpicture} [transform canvas={scale=.7,xshift=-9cm, yshift=.7cm}]

   \foreach \x in {1,2,3,4,5}
   	{\pgfmathparse{int(\x*3.1)}
    	\node[circle,draw, fill=black] (v\x) at (\pgfmathresult,0) {};
   	\node (l\x) at (\pgfmathresult,-.5) {\Large $v_{\x}$};}

   \foreach \x/\y/\a in {1/2/.6,2/3/.6,3/4/.6,4/5/.6}
   \draw[thick, decoration={markings,mark=at position \a with
   	{\arrow[scale=3,>=stealth]{>}}},postaction={decorate}] (v\x) to (v\y);

   \end{tikzpicture}
   \end{center}
    \caption{A digraph with no watchmans walks.}
    \label{directedpath}
\end{figure}
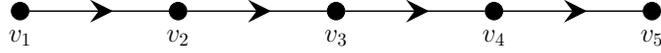

In directed graphs that have strongly connected dominating sets and watchman's walks, $\gamma_{sc}(D)$ and $w(D)$ are not necessarily equal. Consider the directed graph $D$ in Figure \ref{windmill}. In any $\gamma_{sc}$-set, each vertex in the set is dominated by another vertex in the set. Hence, each vertex $v$ in a $\gamma_{sc}$-set dominates at most $|N^+(v)|$ other vertices. It follows that a strongly connected subset of $D$ containing six vertices dominates at most fifteen vertices where $v_7$ must be in any strongly connected dominating set and $u_7$ must not, due to their number of out-neighbours. It can be checked that no dominating set of size six induces a strongly connected subset. Thus, $\gamma_{sc}(D)=7$, and the set $S=\{v_1,v_2, \ldots, v_7\}$ is a $\gamma_{sc}$-set. The shortest closed walk containing all of these vertices has length $9$. The walk of length eight illustrated by dashed arcs is a closed dominating walk in the graph, and hence $w(D)\leq 8$. Thus, no watchman's walk in this graph contains the $\gamma_{sc}$-set $S$ as a subset of its vertices and, in fact, the watchman's walk illustrated by the dashed arcs is completely disjoint from the $\gamma_{sc}$-set $S$. Additionally, $\{v_7, u_1, u_3, u_5\}$ is a $\gamma$-set, so $\gamma(G)=4$.

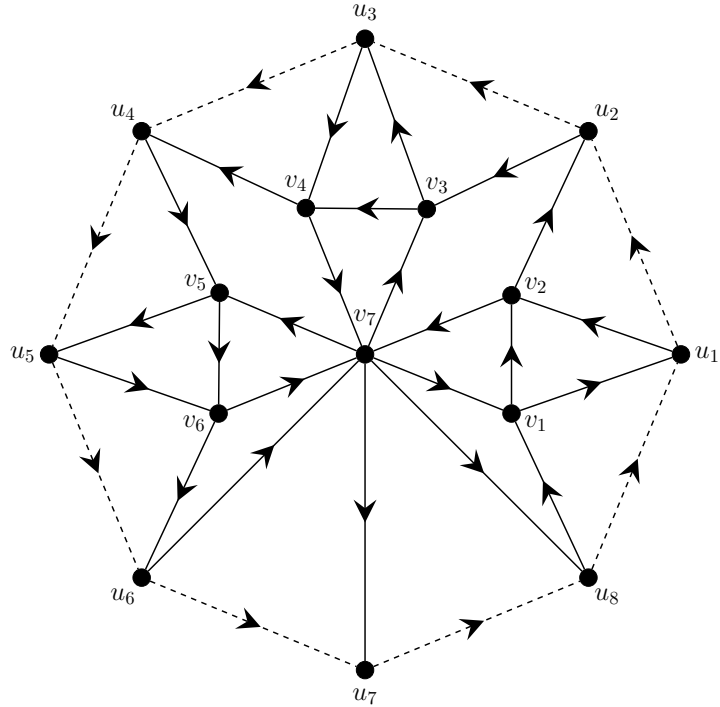
\begin{figure}[htb]
    \centering\vspace*{9.5cm}
    \begin{tikzpicture} [transform canvas={scale=.7,yshift=6.7cm}]

    \foreach \x in {1,2,3,4,5,6}
    {\pgfmathparse{int(\x*360/8-67.5)}
    	\node[circle,draw, fill=black] (v\x) at (\pgfmathresult:3) {};
    	\node (l\x) at (\pgfmathresult:3.5) {\Large $v_{\x}$};}
    \node[circle,draw, fill=black] (v7) at (0,0) {};
    \node (l7) at (0,.7) {\Large $v_7$};
    \foreach \x in {1,2,3,4,5,6,7,8}
    {\pgfmathparse{int(\x*360/8-45)}
    	\node[circle,draw, fill=black] (u\x) at (\pgfmathresult:6) {};
    	\node (k\x) at (\pgfmathresult:6.5) {\Large $u_{\x}$};}
    \foreach \x/\y/\a in {1/2/.535,2/3/.535,3/4/.535,4/5/.535,5/6/.535,6/7/.535,7/8/.535,8/1/.535}
    \draw[thick, dashed, decoration={markings,mark=at position \a with
    	{\arrow[scale=3,>=stealth]{>}}},postaction={decorate}] (u\x) to (u\y);

    \foreach \x/\y/\a in {7/1/.6,1/2/.6,2/7/.6,7/3/.6,3/4/.6,4/7/.6,7/5/.6,5/6/.6,6/7/.6}
    \draw[thick, decoration={markings,mark=at position \a with
    	{\arrow[scale=3,>=stealth]{>}}},postaction={decorate}] (v\x) to (v\y);

    \foreach \x/\y/\a in {1/1/.535,2/2/.535,3/3/.535,4/4/.535,5/5/.535,6/6/.535,7/7/.535,7/8/.535}
    \draw[thick, decoration={markings,mark=at position \a with
    	{\arrow[scale=3,>=stealth]{>}}},postaction={decorate}] (v\x) to (u\y);

    \foreach \x/\y/\a in {1/2/.6,2/3/.6,3/4/.6,4/5/.6,5/6/.6,6/7/.6,8/1/.6}
    \draw[thick, decoration={markings,mark=at position \a with
    	{\arrow[scale=3,>=stealth]{>}}},postaction={decorate}] (u\x) to (v\y);
    \end{tikzpicture}
    \caption{A digraph with disjoint $\gamma_{sc}$-set and watchman's walk}
    \label{windmill}
\end{figure}

A directed graph can be formed from an undirected simple graph $G$ by assigning a direction to each edge in $G$. The resulting graph is called an \textit{orientation} of $G$. A \textit{tournament} is an orientation of a complete graph. A \textit{strongly connected component} of a digraph $D$ is a subdigraph of $D$ that is also strongly connected. The \textit{condensation} of a digraph $D$, denoted by $D^*$, is the digraph found by contracting each maximal strongly connected component to a single vertex. In this digraph, there is an arc from vertex $W$ to vertex $U$ if all arcs between strongly connected components $W$ and $U$ in $T$ are directed from a vertex in $W$ to a vertex in $U$.

The results in this paper originally appeared as part of \cite{britthesis}.

\section{Watching directed graphs}

Unlike an undirected graph, a digraph may not have a watchman's walk. Thus, we begin with results on the existence of watchman's walks in digraphs, before moving on to families where we can guarantee a watchman's walk exist.
 
\subsection{Watchman's walks in general digraphs}

 In the following theorem, we generalize the result in Corollary \ref{condww} for tournaments to general digraphs.

\begin{theorem} \label{strongdigww}
If $D$ is a digraph with a strongly connected subdigraph $D'$ such that the vertices of $D'$ are a dominating set in $D$, then $D$ has a watchman's walk.
\end{theorem}

\begin{proof}
Let $D'$ be a strongly connected subdigraph of $D$ with its vertices forming a dominating set in $D$. So there exists a directed path between any pair of vertices in $D'$. This implies there exists a closed walk $W$ that passes through every vertex in $D'$, possibly repeating some vertices. Clearly $W$ is a dominating walk in $D'$ and since the vertices of $D'$ dominate $D$, $W$ is also a dominating walk in $D$. $W$ is not necessarily of minimal length but is a closed dominating walk, and therefore $D$ has a minimum closed dominating walk.
\end{proof}

We now know that any strongly connected digraph has a watchman's walk, regardless of its structure. For the remainder of this section, we consider digraphs that are not necessarily strongly connected. Recall that a source vertex in a digraph is a vertex with no in-neighbours. Any nontrivial digraph with a source vertex is not strongly connected, as there is no nontrivial walk that ends at a source vertex. It is easy to see that a digraph that has exactly one source vertex has a watchman's walk if and only if that source vertex dominates all the other vertices. Also, if a digraph has more than one source vertex it cannot have a watchman's walk. Thus, we get the following theorem.

\begin{theorem}
A digraph $D$ of order $n$ with $k\geq 1$ source vertices has a watchman's walk if and only if $D$ has exactly one source vertex $v$, where $\deg^+(v)=n-1$.
\end{theorem}

If we take an orientation of any path on at least four vertices, we get a digraph that is not strongly connected and, in fact, no such orientation has a watchman's walk. However, there does exist an orientation of a $n$-cycle will have a watchman's walk; one in which all the arcs of are oriented as a directed $n$-cycle. Thus there are only two such orientations.

\subsection{Orientations of complete multipartite graphs}

In this section, we will consider the existence of watchman's walks in orientations of complete bipartite and multipartite graphs. Unlike in tournaments, these do not always have a watchman's walk. Consider, for example, a complete multipartite graph with more than one source vertex. In the following observation, we consider the existence of watchman's walks in multipartite graphs with at least one source vertex.

\begin{observation}
Let $D=(X_1\cup \ldots \cup X_k, A)$ be an orientation of a complete multipartite graph for $k\geq 2$, where $X_1,\ldots, X_k$ are the vertex sets of the partition. If $v\in X_i$ is a source vertex for any $1\leq i \leq k$, then $D$ has a watchman's walk if and only if $|X_i|=1$.
\end{observation}

 We move on to orientations that have no source vertices. We begin by considering orientations of complete bipartite graphs.

\begin{theorem} \label{bipwwthm}
Let $D$ be an orientation of a complete bipartite graph with partition $(A,B)$ such that $\delta^-(D)\geq 1$. If there exists a $U\subseteq B$ such that $D[A\cup U]$ is strongly connected, then $D$ has a watchman's walk.
\end{theorem}

\begin{proof}
	First note that since $\delta^-(D)\geq 1$, $A$ dominates every vertex in $B$. This means that $D[A\cup U]$ is strongly connected and dominating. By Theorem~\ref{strongdigww}, $D$ has a watchman's walk.
\end{proof}

\begin{corollary} \label{bipww}
Let $D$ be an orientation of a complete bipartite graph with partition $(A,B)$ such that $\delta^-(D)\geq 1$ and $|A|\leq |B|$. If there is a set of vertices $U\subseteq B$ such that each vertex in $A$ is dominated by exactly one vertex in $U$, then $D$ has a watchman's walk. Moreover, $w(D)\leq 2|A|$.
\end{corollary}

\begin{proof}
Let $D$ be an orientation of a $K_{m,n}$ such that each vertex has at least one in-neighbour. Let $V(D)=A\cup B$, such that $|A|=m$ and  $|B|=n$. Let $A=\{v_1,v_2, \ldots,v_m\}$. Suppose $U$ is a set as defined in the statement of the corollary. Clearly $D[A\cup U]$ is strongly connected and dominating, so by Theorem~\ref{bipwwthm}, $D$ has a watchman's walk. We will exhibit a closed dominating walk of length $2|A|$.

By definition, each vertex in $U$ has a private out-neighbour in $A$. Thus, $|U|=m$; let $U=\{u_1,u_2, \ldots, u_m\}$. Label the vertices of $U$ such that vertex $u_i$ dominates $v_i$. It follows that $v_i$ dominates each vertex in $U\backslash \{u_i\}$. Thus, we can consider the closed walk $W=\{u_1, v_1, u_2, v_2, \ldots, u_m, v_m, u_1\}$. The walk $W$ uses every vertex in $A$ exactly once, and hence, $W$ is a closed dominating walk of length $2|A|$ in $D$.
\end{proof}

Theorem~\ref{bipwwthm} and Corollary~\ref{bipww} tells us that there exists orientations of complete bipartite graphs that have a watchman's walk. We generalize this result to orientations of complete multipartite graphs that do not contain any source vertices.

\begin{theorem}
Let $D$ be an orientation of a complete $k$-partite graph for $k>1$. If $\delta^-(D)\geq 1$, then the condensation of $D$ has a source vertex, and $D$ has a watchman's walk.
\end{theorem}

\begin{proof}
Let $D$ be an orientation of a complete multipartite graph such that $\delta^-(D)\geq 1$. That is, each vertex in $D$ has at least one in-neighbour. Let $\{T_1, T_2, \ldots T_k\}$ be the maximal strongly connected components of $D$. Consider the condensation, $D^*$. Let the vertices of $D^*$ be $\{t_1, t_2, \ldots, t_k\}$, where vertex $t_i$ in $D^*$ corresponds to the component $T_i$ in $D$. Let $t$ be any vertex in $D^*$, and let $P$ be a maximal path ending at $t$. Consider the start vertex $u$ of $P$. Any cycle in $D^*$ would correspond to a larger strongly connected component in $D^*$, contradicting the maximality of each $T_i$. Hence, $u$ has no in-neighbours on $P$. Since $P$ was maximal, $u$ also does not have any in-neighbours in the subdigraph $D^* \backslash P$. 
This means that, in $D$, any vertex in the component $U$ corresponding to the vertex $u \in D^*$ has no in-neighbours from any vertex in another strongly connected component. As each vertex must have at least one in-neighbour, $U$ must contain more than one vertex. Hence, $U$ contains vertices from more than one set in the partition of $V$. As $D$ is an orientation of a complete multipartite graph, it follows that there is at least one arc between a vertex in $U$ and a vertex in each set in the partition of $V$, and hence an arc between $U$ and every other maximal strong component. This guarantees that there is an arc between $u$ in $D^*$ and each other vertex in $D^*$. Since $u$ is a source in $D^*$, the component $U$ in $D$ is dominating. As $U$ is a strongly connected component, there is a closed walk containing all of the vertices in $U$. As $U$ is a dominating set, this walk is a closed dominating walk in $D$. Therefore, there is a watchman's walk for $D$.
\end{proof}
\section {Watchman's walks in tournaments}

 If a tournament is strongly connected, we call it a strong tournament. In this section, we refer to classical results, primarily those found in \cite{handbook} and \cite{algorithms}, to prove the existence of a watchman's walk in strong tournaments, and later, we generalize our result for tournaments that are not necessarily strong. We also provide bounds on the length of a watchman's walk in tournaments. We begin by stating a theorem that will be useful in the proof of subsequent results.

\begin{theorem} {\em\cite{semi}} \label{stronghamilton}
A tournament is Hamiltonian if and only if it is strongly connected.
\end{theorem}

\begin{theorem}\label{strongww}
If $T$ is a strong tournament, then $T$ has a watchman's walk.
\end{theorem}

\begin{proof}
If $T$ is a strong tournament then, by Theorem \ref{stronghamilton}, $T$ has a Hamilton cycle $H$. Since every vertex is on this cycle, this is a closed dominating walk. This means that the set of closed dominating walks of $T$ is non-empty. So, a minimum closed dominating walk of $T$ exists.
\end{proof}

Theorem \ref{strongww} implies that any strong tournament on $n$ vertices has a closed dominating walk of length $n$. The question arises, do all vertices in a tournament would need to be included in a watchman's walk. Theorem \ref{nminustwo} below asserts that we can always find a closed dominating walk with no more than $n-2$ vertices for tournaments of order at least $5$. We first state the following theorem as a useful tool in the proof of Theorem \ref{nminustwo}.

\begin{theorem}{\em\cite{hamiltonian}}\label{cycles}
Every vertex in a strong tournament of order $n$ is contained in a cycle of length $k$ for $k=3,4,\ldots,n $.
\end{theorem}

\begin{theorem}\label{nminustwo}
If $T$ is a strong tournament of order $n \geq 5$, then $w(T)\leq n-2$.
\end{theorem}

\begin{proof}

Let $T$ be a strong tournament of order $n\geq 5$. From Theorem \ref{cycles}, there is some cycle of length $n-2$ in $T$, call it $C$. Consider the vertices $u$ and $v$ that are not in $C$; without loss of generality, let $uv$ be an arc. If $C$ dominates $\{u,v\}$ then $C$ is a closed dominating walk and hence $w(T)\leq n-2$. If $C$ does not dominate $\{u,v\}$ then since $T$ is strongly connected, there must be at least one arc from a vertex in $C$ to $u$, and no arc from a vertex in $C$ to $v$. Let $x$ be a vertex in $C$ such that $xu$ is an arc. We have that $W=x,u,v,x$ is a closed dominating walk of length $3\leq n-2$. Thus the result holds. \end{proof}

The following theorem draws a connection between the watchman's walk of a spanning subdigraph and the tournament itself.

\begin{theorem} \label{spanningww}
Let $T$ be a tournament and $T'$ be a spanning subdigraph of $T$. If $T'$ has a watchman's walk, then $w(T)\leq w(T')$.
\end{theorem}

\begin{proof}
Let $T'$ be a subdigraph of $T$ such that $V(T')=V(T)$. Suppose $W$ is a watchman's walk in $T'$. Since $T'$ is a subdigraph of $T$, $W$ is also a walk in $T$. Moreover, $V(T')=V(T)$ and $W$ is a dominating walk in $T'$, so $W$ is also a dominating walk in $T$. Thus, any watchman's walk $W$ in $T'$ is also closed dominating walk in $T$. However, there may be a shorter closed dominating walk in $T$. Therefore, $w(T)\leq w(T')$.
\end{proof}

If we have a spanning subdigraph $T'$ in $T$, we cannot assume that $T'$ has a watchman's walk, since $T'$ is a digraph. However, if $T'$ has a watchman's walk, Theorem~\ref{spanningww} shows that the length of a watchman's walk in $T$ is bounded above by $w(T')$. Moreover, if we have a (not necessarily spanning) subdigraph $T'$ such that the vertices of $T'$ are a dominating set in $T$, we know that $w(T)$ is at most $w(T')$. This fact is a strengthening of Theorem~\ref{spanningww}, and is proved in the following theorems.

Note that a tournament $T=(V,A)$ is transitive if for any vertices $a,b$, and $c$, $(a,b)\in A$, and $(b,c)\in A$ implies that $(a,c)\in A$. From \cite{condtrans}, we get the following important theorem.

\begin{theorem}{\em\cite{condtrans}}\label{condtransitive}
The condensation of any tournament is a transitive tournament.
\end{theorem}

We show below that Theorem~\ref{condtransitive} limits the possible vertices that may occur in a watchman's walk. Given a tournament $T$, where $T'$ is the maximal strong component that corresponds to the dominating vertex in the condensation of $T$, we call $T'$ the \textit{dominating strong component}. Since any transitive tournament has a dominating vertex, and the condensation of a tournament is a transitive tournament, the dominating strong component in a tournament always exists.

\begin{theorem} \label{dominationcond}
Let $T$ be a tournament. If $T'$ is the dominating strong component in $T$, then $w(T)=w(T')$.
\end{theorem}

\begin{proof}
Let $T$ be a tournament. If $T$ is a strong tournament, then $T^*$ has exactly one component, and this component is $T$ itself. In this case, the result is clear.

Now suppose that $T$ is not strong. This means, by Theorem \ref{condtransitive}, $T^*$ is transitive with more than one component. Let $t'$ be the dominating vertex of $T^*$, which corresponds to the dominating strong component of $T$. Since $T^*$ is transitive, $t'$ is a dominating vertex. This means in $T$, there are no arcs from $V(T)\backslash V(T')$ to $V(T')$. Thus, a watchman's walk in $T'$ is a closed dominating walk in $T$. It is also minimal as any shorter walk that included a vertex from $V(T)\backslash V(T')$ would not be closed, and a shorter walk completed inside $T'$ would not be dominating. Thus $w(T)=w(T')$.

\end{proof}

From the proof of Theorem~\ref{condtransitive}, it is clear that the watchman's walk in a tournament is always contained in the dominating strong component. This result is stated below in Corollary~\ref{condww}. It also leads us to an important result in Theorem~\ref{allww}.

\begin{corollary} \label{condww}
If $T$ is a tournament, and $T'$ is the dominating component in the $T^*$, then all watchman's walks for $T$ are contained in $T'$.
\end{corollary}

The fact that any tournament has a watchman's walk is not immediately obvious for tournaments that are not strong, as many digraphs that are not strongly connected do not have a watchman's walk. In fact, there are many infinite families of digraphs that have no watchman's walk, including the family of orientations of paths on $n\geq 4$ vertices.

\begin{theorem}\label{allww}
If $T$ is a tournament, then $T$ has a watchman's walk.
\end{theorem}

\begin{proof}
Let $T$ be a tournament. If $T$ is strong, then by Theorem~\ref{strongww}, $T$ has a watchman's walk. Suppose that $T$ is not strong, and consider the dominating vertex in $T^*$. This vertex corresponds to the dominating maximal strong component in $T$. This component $T'$ is strongly connected, so the subtournament $T'$ has a watchman's walk. By Corollary~\ref{condww}, a watchman's walk of $T$, if it exists, is a watchman's walk of $T'$. Since $T'$ has a watchman's walk, $T$ also has a watchman's walk.
\end{proof}

A \textit{semicomplete digraph} is a digraph in which there is at least one arc between any pair of vertices. The previous theorem can be extended to semicomplete digraphs with only minor changes to the proof given above.

\begin{corollary}
If $D$ is a semicomplete digraph, then $D$ has a watchman's walk.
\end{corollary}

Using Theorem~\ref{allww}, we can extend the bound in Theorem~\ref{nminustwo} from strong tournaments to tournaments in general to get the following bound.

\begin{corollary}
If $T$ is a tournament of order $n\geq 5$, then $w(T)\leq n-2$.
\end{corollary}

The upper bound given above is best possible for tournaments of small order. Consider a tournament on $5$ vertices that does not have a source vertex. In this case, $w(T)\geq 3$.

\section{Domination number}

\subsection{Watchman's walk in relation to domination number}
In general, for both undirected and directed graphs, there can be a large difference between the domination number and the length of a watchman's walk in a graph. In this section, we show that this is not the case for tournaments. We begin by proving a relationship between the domination number and size of a watchman's walk in a tournament. We later give further results relating to domination number and variants.

\begin{theorem} {\em\cite{redei}} \label{hampath}
Every tournament contains a Hamilton path.
\end{theorem}

\begin{theorem} \label{domset}
Let $T$ be a tournament of order $n\geq 3$. If $\gamma(T)> 1$, then $w(T)=\gamma(T)$ or $w(T)=\gamma(T) + 1.$
\end{theorem}

\begin{proof}
Let $T$ be a tournament of order $n$. Let $\gamma(T)=k$ for some $1 \leq k < n$. From Theorem \ref{hampath}, any tournament contains a Hamilton path. Suppose that there is some minimum dominating set $D=\{v_1,v_2,v_3,\ldots,v_k\}$ of $T$, with a Hamilton path $H=v_1,v_2,v_3,\ldots,v_k$ in the subtournament induced by $D$ such that $v_1 \in N^+_T(v_k)$. In this case $(v_k,v_1)$ is an arc, so $W=v_1,v_2$, $v_3,\ldots,v_k,v_1$ is a closed walk in $T$. Since $D$ was a minimum dominating set in $T$, and $W$ is a closed walk that uses exactly the vertices of $D$, $W$ is a watchman's walk for $T$. This walk has has length $k$. So, $w(T)=k=\gamma(T).$

Now suppose that every minimum dominating set in $T$ induces a subtournament that does not have a Hamilton cycle. It follows that we cannot construct a closed dominating walk in $T$ that is of length $\gamma(T)=k$, as this walk must use only the vertices of a minimum dominating set, and there are no closed walks using exactly the vertices of a minimum dominating set. This means that $w(T)>k=\gamma(T)$. Let $D=\{v_1,v_2,v_3,\ldots,v_k\}$ be a minimum dominating set in $T$, and let $P=v_1,v_2,v_3,\ldots,v_k$ be a Hamilton path in $H=T\lbrack D\rbrack $. Consider $ N^+_T(v_k) \cap N^-_T(v_1).$ If this is empty, every vertex outside of $D$ that is dominated by $v_k$ is also dominated by $v_1$, and $v_k$ is dominated by $v_1$. So, $D\backslash \{v_k\}$ is a dominating set, contradicting the minimality of $D$. So, there is some vertex $u$ in $T$, such that $u \not\in D$, and $u \in N^+_T(v_k) \cap N^-_T(v_1)$. Thus, $W=v_1,v_2,v_3,\ldots,v_k,u,v_1$ is a closed dominating walk of length $k+1=\gamma(T)+1$. Since $W$ has length $k+1$ it is a minimum closed dominating walk and so, $w(T)=\gamma(T)+1$. Therefore, for any tournament $T$, $w(T)=\gamma (T)$ or $\gamma (T) +1$. \end{proof}

In the tournament in Figure \ref{wwpaley7}, $\gamma(T)=3$, and $\{v_1,v_2,v_6\}$ is a minimum dominating set. Since $W=v_1,v_2,v_6,v_1$ is a closed walk through this set, $w(T)=3=\gamma(T)$. In the tournament in Figure~\ref{unique14}, $\gamma(T) = 2$, $\{v_1,v_5\}$ is minimum dominating set, and $W = v_1,v_2,v_5,v_1$ is watchman's walk, so $w(T) = 3 =\gamma(T) + 1$.

\begin{figure}[h]
    \centering
    \begin{subfigure}[b]{0.45\textwidth}\vspace*{5.4cm}
   \centering
    	\begin{tikzpicture}[transform canvas={scale=.6, yshift=4.6cm}]

    \foreach \x in {1,2,3,4,5,6,7}
    {\pgfmathparse{int(\x*360/7)}
    	\node[circle,draw, fill=black] (v\x) at (\pgfmathresult:4) {};
    	\node (l\x) at (\pgfmathresult:4.5) {\LARGE $v_{\x}$};}

    \foreach \x/\y/\a in {1/3/.535,1/5/.535,2/3/.535,2/4/.535,3/4/.535,3/5/.535,3/7/.535,4/5/.535,4/6/.535,5/6/.535,5/7/.535,6/7/.535}
    \draw[thick, decoration={markings,mark=at position \a with
    	{\arrow[scale=3,>=stealth]{>}}},postaction={decorate}] (v\x) to (v\y);

    \foreach \x/\y/\a in {1/4/.535,1/7/.535,2/5/.535,2/7/.535,3/6/.535,4/7/.535}
    \draw[thick, decoration={markings,mark=at position \a with
    	{\arrow[scale=3,>=stealth]{>}}},postaction={decorate}] (v\y) to (v\x);

    \foreach \x/\y/\a in {1/2/.535,6/1/.535,2/6/.535}
    \draw[thick, dashed, decoration={markings,mark=at position \a with
    	{\arrow[scale=3,>=stealth]{>}}},postaction={decorate}] (v\x) to (v\y);
    \end{tikzpicture}
    \caption{A tournament $T$ where \\ $w(T)=\gamma(T)$.}
    \label{wwpaley7}
    \end{subfigure}\hfill
    \begin{subfigure}[b]{0.45\textwidth}
   \centering
    	\begin{tikzpicture} [transform canvas={scale=.6, yshift=4.6cm}]

    \foreach \x in {1,2,3,4,5,6,7}
    {\pgfmathparse{int(\x*360/7)}
    	\node[circle,draw, fill=black] (v\x) at (\pgfmathresult:4) {};
    	\node (l\x) at (\pgfmathresult:4.5) {\LARGE $v_{\x}$};}

    \foreach \x/\y/\a in {1/2/.535,1/3/.535,1/4/.535,2/3/.535,2/4/.535,2/5/.535,3/4/.535,3/5/.535,3/6/.535,4/5/.535,4/6/.535,4/7/.535,5/6/.535,5/7/.535,6/7/.535}
    \draw[thick, decoration={markings,mark=at position \a with
    	{\arrow[scale=3,>=stealth]{>}}},postaction={decorate}] (v\x) to (v\y);

    \foreach \x/\y/\a in {1/5/.535,1/6/.535,1/7/.535,2/6/.535,2/7/.535,3/7/.535}
    \draw[thick, decoration={markings,mark=at position \a with
    	{\arrow[scale=3,>=stealth]{>}}},postaction={decorate}] (v\y) to (v\x);

    \end{tikzpicture}
    \caption{Tournament of order $7$ with $14$ watchman's walks}
    \label{unique14}
    \end{subfigure}
    \caption{Tournaments of order $7$.}
\end{figure}
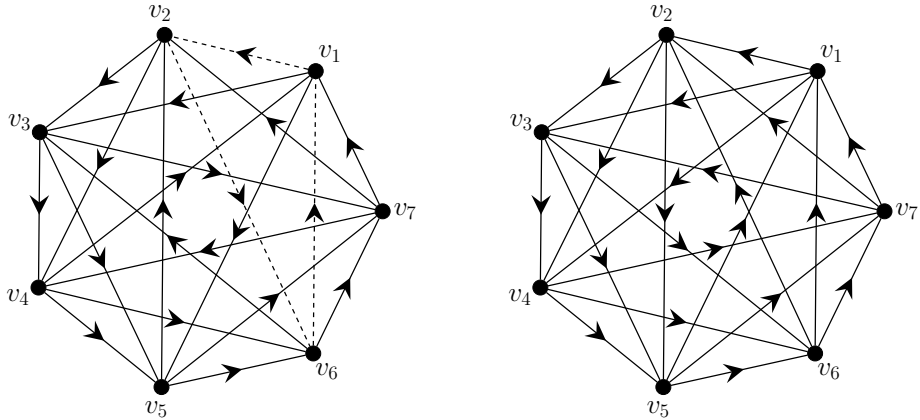

The result in Theorem \ref{domset} can also be extended to semicomplete digraphs.

\begin{corollary} \label{scdomset}
Let $D$ be a semicomplete digraph of order $n\geq 2$. If $\gamma(D)$ is the size of a minimum dominating set and $\gamma(D)> 1$, then $w(D)=\gamma(D)$ or $\gamma(D) + 1$.
\end{corollary}

From Theorem \ref{domset}, we get the following two corollaries.

\begin{corollary}\label{domsetrepeat}
If $T$ is a tournament, then no watchman's walk in $T$ repeats vertices.
\end{corollary}
 \begin{proof}
 Let $T$ be a tournament. Suppose there is some minimum dominating set $D$ such that the subtournament induced by $D$, $T\lbrack D \rbrack$, contains a Hamilton cycle. In this case, $w(T)=|D|=\gamma(T)$. This means that any watchman's walk must contain the vertices of a minimum dominating set, since no other set of size $\gamma(T)$ will be dominating, and any larger dominating set will not give us a closed walk of length $\gamma(T)$. Since any minimum dominating set contains only distinct vertices, each watchman's walk for $T$ does not repeat any vertices.

Now suppose that there is no minimum dominating set such that the subtournament it induces contains a Hamilton cycle. In this case, $w(T)=\gamma(T)+1$, as we have no closed dominating walk containing exactly the vertices of a minimum dominating set. Let $W$ be any watchman's walk of $T$, and let $D$ be the set of vertices in the walk. Since $D$ must be a dominating set of $T$, we have that $D$ is either a minimal dominating set of size $\gamma(T)+1$, or a minimum dominating set with an additional vertex. Note that it cannot be a dominating set of size $\gamma(T)$ with one of those vertices used twice as that would mean the dominating set was strongly connected and hence would have a Hamilton cycle. In the first case, each vertex must be distinct in the walk as the length of the walk equals the number of vertices. In the second case, each of the vertices from the dominating set occurs once and the additional vertex is from outside the dominating set, thus the length of the walk equals the number of vertices and the result holds. 
\end{proof}

\begin{corollary} \label{whenwTk}
If $T$ is a tournament, then $\gamma(T)=w(T)$ exactly when there exists some minimum dominating set $D$ in $T$ such that $T[D]$ is itself a strong tournament.
\end{corollary}

\begin{proof}
Let $T$ be a tournament such that $\gamma(T)=k$, and suppose that for some minimum dominating set $D$, $T\lbrack D \rbrack$ is strongly connected. It follows, by Theorem \ref{stronghamilton}, that $T\lbrack D \rbrack$ has a Hamilton cycle $H$ and this is also a closed dominating walk in $T$. By Theorem~\ref{domset}, since $w(T)\geq \gamma(T)$. This must be minimum.

Now suppose that for every minimum dominating set $D$ in $T$, $T\lbrack D \rbrack$ is not a strong subtournament. This means there is no Hamilton cycle in $T\lbrack D \rbrack$. So, there is no $k$-cycle in $T$ containing exactly the vertices of a minimum dominating set. Hence, any cycle of length $k$ does not dominate $T$. Thus, any closed walk of length $k$ is not dominating, and $w(T)\neq k$. By Theorem \ref{domset}, $w(T)=k$ or $k+1$, so in this case, $w(T)=k+1= \gamma(T)+1$. Therefore, $w(T)=\gamma(T)$ exactly when $T\lbrack D \rbrack$ is strong for some minimum dominating set $D$.
\end{proof}

The corollaries given of Theorem \ref{domset} can also be generalized to semicomplete digraphs. In particular, we have the following corollary, the proof of which can be generalized using Theorem \ref{semihamcycle}.

\begin{theorem}\em{\cite{semi}} \label{semihamcycle}
Every strongly connected semicomplete digraph has a Hamilton cycle.
\end{theorem}

\begin{corollary} \label{norepeats}
If $T$ is a semicomplete digraph, then no watchman's walk repeats any vertices.
\end{corollary}

The cycle domination number of a digraph $D$,  $\gamma_{\text{cyc}}(D)$,  is the length of the shortest directed cycle $C$ in $D$ such that the vertices in $C$ form a dominating set in $D$. By Corollary~\ref{domsetrepeat}, a watchman's walk in a tournament never repeats a vertex and hence is a dominating cycle of minimum length. We get the following theorem.

\begin{theorem}
If $T$ is a tournament, then $\gamma_{\text{cyc}}(T)=w(T)$.
\end{theorem}

The total domination number of a digraph $D$, denoted $\gamma_{\text{t}}(D)$, is the size of the smallest set $S$ such that for every vertex $v$ in $D$, there is a vertex $u\not= v$ in $S$ such that $(u,v)$ is an arc.

\begin{theorem}
If $T$ is a tournament, then $\gamma(T) \leq \gamma_{\text{t}}(T)\leq w(T)$.
\end{theorem}

\begin{proof}
Any minimum total dominating set must also be a dominating set, so $\gamma(T) \leq \gamma_{\text{t}}(T)$. If $W$ is a watchman's walk, each vertex in $W$ is dominated by at least one other vertex in the walk, and by Corollary \ref{domsetrepeat}, the vertices of $W$ are unique. The vertices of $W$ form a dominating set in $T$, so the vertices of $W$ are also a total dominating set in $T$. Therefore, $\gamma(T) \leq \gamma_{\text{t}}(T)\leq w(T)$.
\end{proof}
\
Note that the total domination number is not equal to either the domination number or the watchman number. In Figure~\ref{unique14},  $w(T)=\gamma_{\text{t}}(T)=3=\gamma(T)+1$. The smallest example the authors could find with $w(T)<\gamma_{\text{t}}(T)=\gamma(T)$ has $w(T)=5$, $\gamma_{\text{t}}(T)=\gamma(T)=4$, and has $272$ vertices.We give a construction for of such a tournament at the end of the next section.

\section{Small and fixed watchman number}

For some families of tournaments, particularly those that are highly structured or small order, the watchman number is fixed or bounded. The following theorem, from \cite{irredundance}, provides bounds on the domination number for small tournaments.

\begin{theorem} {\em\cite{irredundance}}\label{dom7}
If $T$ is a tournament on $n$ vertices, then $\gamma(T)\leq 2$ if $n<7$, and $\gamma(T)\leq 3$ if $n<19$.
\end{theorem}

By using these upper bounds on the domination number of small tournaments, we get the following upper bounds on the watchman's walk number.

\begin{corollary} \label{bound7}
If $T$ is a tournament on $n$ vertices, then $w(T)\leq 3$ if $n<7$, and $w(T)\leq 4$ if $n<19$.
\end{corollary}

We know that any watchman's walk for a tournament will be completely contained in the dominating strong component. From this, it follows that we can also apply these upper bounds to some larger tournaments.

\begin{corollary}
If $T$ is a tournament such that the dominating strong component contains fewer than $7$ vertices, then $w(T)\leq 3$.
\end{corollary}

\begin{proof}
Let $T$ be a tournament with the dominating strong component containing fewer than $7$ vertices. If $T$ is strongly connected, then $T$ has exactly one strongly connected component, so $|V|<7$, and the result follows from Corollary \ref{bound7}. Now suppose that $T$ is not strongly connected. That is, $T$ contains at least two maximal strong components. Since $T$ is a tournament, by Theorem \ref{dominationcond}, $w(T)=w(T')$, where $T'$ is the dominating strong component. Since $T'$ contains less than seven vertices, $w(T')\leq 3$. Therefore $w(T)\leq 3.$
\end{proof}

Similarly, from the second upper bound, we get the following corollary.

\begin{corollary}
If $T$ is a tournament such that the dominating strong component contains fewer than $19$ vertices, then $w(T)\leq 4$.
\end{corollary}

From \cite{exactlyk}, we know that there exists a tournament with domination number $k$ for any arbitrarily large value of $k$. This paper gives a construction, which when given the inheritance graph of a tournament with domination number larger than $k$, and any tournament $X$ such that $X-M$ dominates $M$ for a specified subset $M$, produces a tournament with domination number $k$ and $|M|$ minimum dominating sets. For many uses and constructions, it is more helpful to have a strong tournament over one that is not strong. While \cite{exactlyk} gives us a tournament $T$ with domination number $k$, this theorem does not guarantee to give a strong tournament. We give the following construction to build a strong tournament from any tournament that is not already strong, while preserving the dominating number. This can be used to obtain a strong tournament with domination number $k$ for any integer $k$.

\begin{Lemma}\label{weaktostrong}
Let $T$ be any tournament of order $n$ that is not strongly connected. If $\gamma (T) = k \geq 3$, then we can construct a strongly connected tournament, $T_s$, of order $n+1$ such that $\gamma(T_s)=k$ and $T$ is a subtournament of $T_s$.
\end{Lemma}
\begin{proof}
Let $T=(V,A)$ be a tournament that is not strongly connected. Since $T$ is a tournament, we known from Theorem \ref{hampath} that $T$ has at least one Hamilton path. Let $H=u,v_1,v_2, \ldots, v_{n-1}$ be a Hamilton path in $T$. Consider the tournament $T_s$, where $V(T_s)=V(T) \cup \{w\}$, and $A(T_s)=A(T) \cup \{(x,w)| x\in V(T)\backslash\{u\}\} \cup \{(w,u)\}$ for some $w \not \in V(T)$. It follows that $H_s=u,v_1,v_2,\ldots,$ $v_{n-1},$ $w,u$ is a closed walk in $T_s$. Since $H$ uses each vertex in $T$ exactly once, $H_s$ uses each vertex in $T_s$ exactly once, except $u$, which is both the start-vertex and end-vertex of the walk. Thus, $H_s$ is a Hamilton cycle in $T_s$. Hence, by Theorem \ref{stronghamilton}, $T_s$ is strongly connected.

Now suppose $D$ is a minimum dominating set in $T$. Since $\gamma(T)=k\geq 3$, we know $|D|=k\geq 3$. As $w$ is dominated in $T_s$ by all but one vertex from $T$, at least two vertices in $D$ dominate $w$ in $T_s$. Also, $A(T)\subset A(T_s)$, so $D$ must also be a dominating set in $T_s$. Thus, $\gamma(T_s)\leq k$.

Consider a minimum dominating set $D_s$ in $T_s$. We know $|D_s|\leq k$. Suppose to the contrary that $|D_s|<k$. Suppose that $D_s \subset V(T)$. Since $A(T)\subset A(T_s)$, $D_s$ would be a dominating set in $T$. However, $|D_s|\leq k-1$, so $D_s$ doesn't dominate $T$. Thus, it also does not dominate $T_s$. This means that $D_s \not\subset V(T)$ and $w \in D_s$. By the construction of $T_s$, vertex $w$ only dominates $w$ and $u$ in $T_s$. So, $D_s\backslash \{w\}$ must dominate $T\backslash \{u\}$. It follows that $D_s\backslash \{w\} \cup \{u\}$ must dominate $T$. However, $|D_s\backslash \{w\} \cup \{u\}|\leq k-1$. This contradicts the fact that $\gamma(T)=k$. Thus, $\gamma(T_s)\geq k$. Therefore $\gamma(T_s)=k$. \end{proof}

\begin{theorem} \label{strongforanyk}
For every integer $k \geq 2$, there exists a strong tournament $T$ such that $\gamma(T)=k$.
\end{theorem}

\begin{proof}
Let $k=2$. The directed $3$-cycle is a strong tournament of order $3$ with $\gamma(T)=2$. Now let $k$ be any integer such that $k\geq 3$. From \cite{exactlyk}, 
there exists a tournament $T$ such that $\gamma(T)=k$. If $T$ is strongly connected, we have the desired tournament. However, if $T$ is not strong, by Theorem \ref{weaktostrong}, there exists a tournament $T_s$ from $T$ such that $T$ is strong and $\gamma(T_s)=k$ as required.
\end{proof}

We now use the above, along with the construction in \cite{exactlyk}, to obtain a tournament with any watchman number greater than $3$.

\begin{theorem}\label{wwanyk}
If $k\geq 3$, then there exists a tournament $T$ such that $w(T)=k$.
\end{theorem}

 In the construction of \cite{exactlyk} used in Theorem \ref{wwanyk}, it is possible to vary a parameter $m$ to alter the structure of $T$. In fact, we get the following corollary.

 \begin{corollary}
 For any integers $k\geq 3$, and $m\geq 1$ there exists a tournament $T$ such that $w(T)=k$ and $T$ has exactly $m$ watchman's walks. Moreover, there exists a tournament $T$ such that $T$ has a unique watchman's walk of length $k$.
 \end{corollary}
 
  To construct the tournament with $w(T)=5$, and $\gamma_{\text{t}}(T)=\gamma(T)=4$,  we the construction given in \cite{exactlyk}. Using $k=4$ and $m=1$, where in the construction the tournament $X$, and hence the unique $\gamma$ set, is taken to be the tournament formed from adding a sink to a three-cycle and the ingredient tournament is the quadratic residue tournament on $67$ vertices. This guarantees $w(T)=5$ and $\gamma_{\text{t}}(T)=4$.

\subsection{Computational results}

Even when considering only those of small order, the number of non-isomorphic tournaments on a given number of vertices can be very large. As a result, the length, structure, and multiplicity of watchman's walks can vary greatly, even between tournaments on the same number of vertices. For further results on multiplicity of watchman's walks of undirected graphs, see \cite{multiplicity}. In this section, we present and summarize computer compiled data on the watchman's walks and domination number in tournaments of order up to 10. The table in Appendix A presents a summarized collection of this data. In this table, we specify the order, length of watchman's walks, and domination number, and give the number of tournaments with the stated parameters. This table contains computational data regarding the watchman's walks in all tournaments of order at most $10$, using the collection of adjacency matrices given in \cite{digraphs}.

If we consider the table in Appendix A, we can easily identify some interesting tournaments. There is exactly one tournament of order $7$ with domination number $3$. This is the Paley tournament of order $7$. Let $q=3(\hspace{-0.2cm}\mod 4)$ be a prime power. Consider the finite field of order $q$, $F_q$. The \textit{Paley tournament} is the digraph with vertex set $V = F_q$, where $(a,b)$ is an arc if $b-a \in (F_q)^2$.

There is also a unique tournament $T$ on seven vertices with domination number two such that $T$ has fourteen watchman's walks. This graph is illustrated in Figure \ref{unique14}.

We let $T(n)$ be the set of non-isomorphic tournaments on $n$ vertices. Consider the set of tournaments, $T(n)$. If we add a source vertex to each tournament in this set, we get $T(n)$ non-isomorphic tournaments of order $n+1$, each having a domination number of $1$ and a watchman's walk of length $0$. The total number of tournaments of order $n$ is equal to the number of tournaments of order $n+1$ that have a dominating vertex. Similarly, we can consider a tournament on $n-1$ vertices with domination number $\gamma$, watchman number $w$, and watchman multiplicity $m$, we can add a sink vertex to the tournament to get a new tournament on $n$ vertices with the same parameters. It follows that there are $|T(n)|$ many tournaments on $n+1$ vertices with a watchman number of $0$. If we instead add a sink vertex to a tournament on $n$ vertices, we get a tournament on $n+1$ vertices with the same watchman number, domination number, and watchman multiplicity. So, the number of tournaments with order $n$, watchman number $w$, multiplicity $m$, and domination number $\gamma$ is less than or equal to the number of  tournaments with order $n+1$ with watchman number $w$, multiplicity $m$, and domination number $\gamma$.

\section{Families of tournaments}

As noted in the previously, there is a large number of non-isomorphic tournaments on a given number of vertices, and the Table in Appendix A highlights that these tournaments are structurally very different, and their domination or watchman numbers can vary greatly. For many tournaments, however, we may expect their watchman numbers to be low. For many families, we can prove this to be the case. In this section, we consider families of tournaments or tournaments having certain characteristics, in order to offer more precise results on their watchman number.

\subsection{Simple tournaments}
The \textit{score sequence} of a tournament is the sequence of the vertex out-degrees, typically listed in non-decreasing order. A tournament with score sequence $S$ is said to be \textit{simple} if there are no other  non-isomorphic tournaments with score sequence $S$. That is, $T$ is the unique tournament with score sequence $S$. The following theorems demonstrate that any simple tournament has a small watchman number.

\begin{theorem}{\em \cite{simple}} \label{scoreunique}
A score sequence $S$ is the score sequence of exactly one tournament $T$ if and only if every strong component has score sequence $(0), (1,1,1),$ $(1,1,2,2)$, or $(2,2,2,2,2)$.
\end{theorem}

\begin{theorem}
If a score sequence $S$ is the score sequence of a unique tournament $T$, then $w(T)=0$ or $3$.
\end{theorem}

\begin{proof}
Let $T$ be a tournament with score sequence $S$ such that for any other tournament $T_2$ with score sequence $S$, $T \simeq T_2$. By Theorem \ref{scoreunique}, the dominating strong component of $T$ has score sequence $(0)$, $(1,1,1)$, $(1,1,2,2)$, or $(2,2,2,2,2)$. 
It is clear that the tournament defined by $(0)$ has a dominating vertex, and has a watchman's walk of length $0$. None of the score sequences $(1,1,1)$, $(1,1,2,2)$, or $(2,2,2,2,2)$ define a subtournament with a dominating vertex, and each define a subtournament with a dominating walk of length $3$. Hence $T'$ has a watchman's walk of length $3$ and, by Theorem~\ref{dominationcond}, $w(T)=3$.
\end{proof}

\subsection{Transitive tournaments}

Recall that a tournament $T=(V,A)$ is transitive if there exists an ordering of the vertices of $T$, $(v_1,v_2,\ldots,v_n)$ such that $(v_i,v_j)\in A(T)$ if and only if $i <j$. These tournaments are always acyclic. We can also have tournaments whose structure closely resembles that of a transitive tournament, but do not have such an ordering. Three such families of tournaments are locally transitive, locally-in-transitive, and locally-out-transitive tournaments. A tournament $T$ is \textit {locally transitive} (\textit{locally-in-transitive}, \textit{locally-out-transitive}) if for each vertex $v$ in $T$, the subdigraph induced by in- and out-neighbourhoods (in-neighbourhood, out-neighbourhood) of $v$ is a transitive sub-tournament. 

Recall that any transitive tournament has a dominating vertex. This is not necessarily the case for tournaments that are locally-in-transitive or locally-out-transitive. However, due to their ordering, or similar structure to transitive tournaments, we would expect these tournaments to have small dominating sets, and hence, small watchman's walks. In this section, we show this is the case.

\begin{theorem} \label{loctransleq}
If $T$ is a locally-in-transitive or locally-out-transitive tournament, then $\gamma(T)\leq 3$ and $w(T)\leq 3$.
\end{theorem}

\begin{proof}
Let $T$ be a tournament, and let $\gamma = \gamma(T)$. By Theorem \ref{domset}, $w(T)=\gamma$ or $\gamma+1$. Suppose to the contrary that $\gamma\geq 4$, and $w(T)=\gamma$. Let $W=v_1,v_2,\ldots,$ $ v_{\gamma}, v_1$ be a watchman's walk in $T$, and let $D$ be the set of vertices in $W$. Since $w(T)=\gamma(T)$, $D$ is a minimum dominating set of $T$. Each vertex in $D$ is dominated by at least one other vertex in $D$, namely the vertex that precedes it in the closed walk $W$. So, each vertex in $D$ has at least one private out-neighbour in $V(T \backslash D)$ with respect to the dominating set. Let $p_i$ be a private neighbour of $v_i$ for each $1 \leq i \leq \gamma$. It follows that $v_1,p_1$, and $p_{\gamma -1}$ are all in-neighbours of $v_2$. So, if $T$ is a locally-in-transitive tournament, $T\lbrack \{v_1,p_1,p_{\gamma -1}\} \rbrack $ is a transitive subtournament of $T$. Since $(p_{\gamma -1},v_1)$ and $(v_1,p_1)$ are both arcs, $(p_{\gamma -1},p_1)$ must be an arc in $T$. Now, $v_{\gamma -1}, p_{\gamma -1}$, and $p_1$ are all in-neighbours of $v_{\gamma}$, but $v_{\gamma -1},p_{\gamma -1}, p_1, v_{\gamma -1}$  is a cycle in the in-neighbourhood of $v_{\gamma}$. However, this cannot happen since $T$ is locally-in-transitive. Thus, if $T$ is locally-in-transitive, we cannot have that $\gamma\geq 4$ when $w(T)=\gamma$. Similarly, $v_2,\ldots ,v_{\gamma}$, are all out-neighbours of $p_1$. So, if $T$ is locally-out-transitive, $T\lbrack \{v_2,\ldots ,v_{\gamma}\} \rbrack $ is a transitive subtournament of $T$. Since $v_2,v_3,\ldots,v_{\gamma }$ is a path in this subtournament, $(v_2,v_{\gamma })$ must be an arc. Now, $v_1,v_2$ and $v_{\gamma }$ are all out-neighbours of $p_2$. However, since $(v_2,v_{\gamma })$ and $(v_{\gamma },v_1)$ are both arcs, we have that $v_{\gamma },v_1, v_2, v_{\gamma }$ is a cycle. However, this cannot happen since $T$ is locally-out-transitive. Thus, we have a contradiction and if $T$ is locally-out-transitive then $\gamma(T)\geq 4$ when $w(T)=\gamma(T)$.

Now suppose $\gamma\geq 4$ and $w(T)=\gamma+1$ in $T$. Let $W$ be a watchman's walk in $T$, and let $D$ denote the set of vertices in $W$. Since $W$ is a dominating walk, $D$ is either a minimal dominating set of size $\gamma +1$, or $D$ contains as a proper subset a minimum dominating set.

First consider the case where $D$ is a minimal dominating set. Similar to the case when $w(T)=\gamma$, each vertex in $D$ has at least one private out-neighbour in $V(T \backslash D)$ with respect to the dominating set. Let $p_i$ be a private neighbour of $v_i$ for each $1 \leq i \leq \gamma +1$. Also, $v_2,\ldots ,v_{\gamma}$ are all out-neighbours of $p_1$. So, if $T$ is locally-out-transitive, $T\lbrack \{v_2,\ldots ,v_{\gamma}\} \rbrack $ is a transitive subtournament of $T$. Since $v_2,v_3,\ldots,v_{\gamma}$ is a path in this acyclic subtournament, $(v_2,v_{\gamma})$ must be an arc. Now, $v_1,v_2$ and $v_{\gamma}$ are all out-neighbours of $p_2$. However, we have that $v_{\gamma},v_1, v_2, v_{\gamma}$ is a cycle. This is a contradiction. Similarly, $v_1,p_1$, and $p_{\gamma}$ are all in-neighbours of $v_2$. So, if $T$ is a locally-in-transitive tournament, $T\lbrack \{v_1,p_1,p_{\gamma}\} \rbrack $ is a transitive subtournament of $T$. Since $(p_{\gamma},v_1)$ and $(v_1,p_1)$ are both arcs, $(p_{\gamma -1},p_1)$ must be an arc in $T$. Now, $v_{\gamma-1}, p_{\gamma-1}$, and $p_1$ are all in-neighbours of $v_{\gamma}$, but $v_{\gamma-1},p_{\gamma -1}, p_1, v_{\gamma -1}$  is a cycle in the in-neighbourhood of $v_{\gamma}$. This is a contradiction, since $T$ is locally-in-transitive.

Now suppose that $D$ contains as a subset a minimum dominating set $D'$. Let $H=v_1,v_2,\ldots,v_{\gamma -1},v_{\gamma}$ be a Hamilton path in $T[D']$; note that $T[D']$ does not have a Hamilton cycle so $(v_1, v_{\gamma})$ is an arc. 
Each vertex $v_i \in D\setminus \{v_1\}$ is dominated by $v_{i-1}$, so each vertex in $D'\setminus \{v_1\}$ has at least one private out-neighbour when consider $D$. If $p_{\gamma}$ is a private out-neighbour of $v_{\gamma}$, then $p_{\gamma}$ dominates all vertices in $D'$ except for $v_{\gamma}$. In particular, it dominates $v_1$ and $v_2$. If $v_1$ does not have a private out-neighbour, then $\{v_2,v_3,\ldots, v_{\gamma}, p_{\gamma},v_2\}$ is a closed dominating walk of length $\gamma$. This is a contradiction, as $w(T)=\gamma + 1$. So $v_1$ must have at least one private out-neighbour. Let $p_i$ be a private out-neighbour of $v_i$ for each $1 \leq i \leq \gamma-1$. Vertices $p_{\gamma-1}, v_1,$ and $v_{\gamma-1}$ are all in-neighbours of $v_{\gamma}$, since $(v_1,v_{\gamma})$ is an arc. If $T$ is locally-out-transitive, We know $T\lbrack \{{v_\gamma-1},p_{\gamma-1},v_1\} \rbrack$ is a transitive subtournament of $T$, and $(v_{\gamma-1},p_{\gamma-1})$ and $(p_{\gamma-1},v_1)$ are both arcs, so $(v_{\gamma-1},v_1)$ is an arc. However, $p_1$ is also an in-neighbour of $v_{\gamma}$, so $T\lbrack \{v_{\gamma-1},p_{1},v_1\} \rbrack$ is also a transitive subtournament, and $v_{\gamma-1},p_{1},v_1,v_{\gamma-1}$ is a cycle. So, if $T$ is locally-in-transitive and $\gamma(T)\geq 4$, $w(T)\not=\gamma(T)+1$.

Similarly, $p_{i}$ dominates $v_{j}$ for all $v_j\in D\setminus{v_i}$. Thus, we can see that $T\lbrack \{v_2,\ldots,v_{\gamma }\rbrack$ is a transitive subtournament if $T$ is locally\hyp{}out\hyp{}transitive, since $v_2,$ $\ldots,$ $v_{\gamma }$ are all out-neighbours of $p_1$. So, we know that $(v_i,v_j)$ is an arc for all $1<i<j\leq \gamma $. By considering the out-neighbours of $p_{\gamma }$, which include $v_1,\ldots, v_{\gamma (T)-1}$, it is clear that $(v_1,v_i)$ is also an arc for $2\leq i \leq \gamma (T) -1$. If $(p_2,p_3)$ is an arc, then $v_1,p_3$, and $v_3$ are out-neighbours of $p_2$, then $v_1,v_3,p_3,v_1$ is a cycle in the out-neighbourhood of $p_2$. However, this out-neighbourhood induces a transitive subtournament. Thus, $(p_3,p_2)$ must be an arc. However, $p_2$, $v_2$, and $v_1$ would be out-neighbours of $p_3$, and $v_1,v_2,p_2,v_1$ is a cycle. So, $(p_3,p_2)$ also cannot be an arc if $T$ is locally\hyp{}out\hyp{}transitive. This is a contradiction, since $T$ is a tournament. So, we cannot have that $\gamma(T)\geq 4$ when $w(T)=\gamma(T)+1$ if $T$ is locally-out-transitive.
\end{proof}

 The tournament in Figure \ref{unique14} is locally\hyp{}transitive. However, it is not a transitive tournament as it contains $3$-cycles. Since each locally\hyp{}transitive tournament is both locally-in-transitive and locally\hyp{}out\hyp{}transitive, we get the following corollary.

\begin{corollary}
If $T$ is a locally-transitive tournament, then $\gamma(T) \leq 3$, and $w(T)=0$ or $3$.
\end{corollary}

\section{Conclusion}

There are many further directions that this research could take in the future. In particular, variations of the watchman's walk problem on directed graphs, such as allowing multiple watchmen, could be considered. The most efficient route for a given number of guards is also not currently known in directed graphs. Additionally, for families of directed graphs, the minimum number of guards needed to achieve a given specified maximum unseen time is also not currently known.

There are many other digraphs that lend themselves to the watchman's walk problem. In particular, de Bruijn graphs \cite{tournamentnote} and other digraphs with Hamilton cycles, as they are guaranteed to have a watchman's walk.

\section{Acknowledgements}

Dr. Danny Dyer acknowledges the support of NSERC.

\bibliographystyle{abbrv}
\bibliography{bib}

\newpage
\section*{Appendix A}

The following table specifies the number of tournaments ($N$) with a given watchman’s walk number ($w$), domination number ($\gamma$), and multiplicity of watchman’s walks ($m$) for each order ($n$) up to $10$.

\begin{center}
\small
\setlength{\tabcolsep}{4pt}
\begin{tabular}{|C|C|C|M|}
	\hline
	n & w & \gamma & (m,N)  \\
	\hline
	2 & 0 & 1 & (1,1) \\
	\hline
	\multirow{2}{*}{3} & 0 & 1 & (1,1) \\
	\cline{2-4} & 3 & 2 & (1,1) \\
	\hline
	\multirow{2}{*}{4} & 0 & 1 & (1,2) \\
	\cline{2-4} & 3 & 2 & (1,1),(2,1) \\
	\hline
	\multirow{2}{*}{5} & 0 & 1 & (1,4) \\
	\cline{2-4} & 3 & 2 & (1,1),(2,2),(3,3),(4,1),(5,1)\\
	\hline
	\multirow{2}{*}{6} & 0 & 1 &  (1,12) \\
	\cline{2-4}& 3 & 2 & (1,2),(2,4),(3,10),(4,12),(5,6),(6,8),(8,2)  \\
	\hline
	\multirow{4}{*}{7}& 0 & 1 & (1,56)\\
	\cline{2-4}&\multirow{2}{*}{3} &\multirow{2}{*}{2} & (1,4),
	(2,12),
	(3,38),
	(4,74),
	(5,69),
	(6,63),
	(8,40),
	(7,53),
	(9,26),\\
	&&&
	(11,4),
	(10,11),
	(12,3),
	(13,1),
	(14,1)\\
	\cline{2-4}&3 & 3 & (7,1) \\
	\hline
	\multirow{5}{*}{8}& 0 & 1 &(1,456)\\
	\cline{2-4}&\multirow{3}{*}{3} &\multirow{3}{*}{2} &(1,12),
	(2,48),
	(3,208),
	(4,544),
	(5,770),
	(6,820),
	(7,788),
	(8,892),\\
	&&&
	(9,704),
	(10,657),
	(11,387),
	(12,294),
	(13,114),
	(14,99),
	(15,36),\\
	&&&
	(16,27),
	(17,8),
	(18,9),
	(20,2)\\
	\cline{2-4}&3 & 3 &
	(7,2),
	(8,1),
	(9,1),
	(10,1)\\
	\hline
	\multirow{8}{*}{9}& 0 & 1 &(1,6880)\\
	\cline{2-4}&\multirow{5}{*}{3} &\multirow{5}{*}{2} &
	(1,56),
	(2,296),
	(3,1648),
	(4,5684),
	(5,11125),
	(6,14911),
	(8,18889),\\
	&&&
	(7,15929),
	(9,20493),
	(10,21489),
	(11,19734),
	(12,17157),
	(13,12413),\\
	&&&
	(14,8912),
	(15,6108),
	(16,3884),
	(17,2319),
	(18,1519),
	(19,801),\\
	&&&
	(20,461),
	(21,286),
	(22,147),
	(23,72),
	(24,60),
	(25,19),
	(26,4),
	(27,8),\\
	&&&
	(28,4),
	(29,1),
	(30,1)\\
	\cline{2-4}&\multirow{2}{*}{3} &\multirow{2}{*}{3} & (7,6),
	(8,19),
	(9,48),
	(10,65),
	(11,46),
	(12,22),
	(13,14),
	(14,3),
	(15,1),\\
	&&&
	(18,1),
	(27,1)\\
	\hline
	\multirow{12}{*}{10}& 0 & 1 &(1,191536)\\
	\cline{2-4}&\multirow{7}{*}{3} &\multirow{7}{*}{2} &
	(1,456),
	(2,3040),
	(3,20808),
	(4,90528),
	(5,232866),
	(6,395927),\\
	&&&
	(7,493369),
	(8,590172),
	(9,714023),
	(10,874685),
	(11,952415),\\
	&&&
	(12,1013385),
	(13,933658),
	(14,823741),
	(15,665467),
	(16,527268),\\
	&&&
	(17,377567),
	(18,277459),
	(19,184796),
	(20,126674),
	(21,78776),\\
	&&&
	(22,52721),
	(23,31016),
	(24,21213),
	(25,11643),
	(26,7727),
	(27,4137),\\
	&&&
	(28,2622),
	(29,1437),
	(30,1015),
	(31,400),
	(32,367),
	(33,112),
	(34,121),\\
	&&&
	(35,24),
	(36,47),
	(38,18),
	(40,4)\\
	\cline{2-4}&\multirow{4}{*}{3} &\multirow{4}{*}{3} &
	(5,1),
	(7,45),
	(8,360),
	(9,1603),
	(10,3933),
	(11,5672),
	(12,5752),\\
	&&&
	(13,4869),
	(14,3298),
	(15,2015),
	(16,1176),
	(17,585),
	(18,255),\\
	&&&
	(19,127),
	(20,58),
	(21,22),
	(22,19),
	(23,7),
	(24,2),
	(25,3),
	(26,3),\\
	&&&
	(27,5),
	(29,1),
	(30,2),
	(31,1),
	(36,2)\\
	\hline
\end{tabular}
\end{center}

\end{document}